\documentclass[reqno,10pt]{amsart}
\usepackage{verbatim}  % Needed for the "comment" environment to make LaTeX comments
\usepackage{amsthm, colonequals, wrapfig, multicol,vwcol,}
\usepackage[utf8]{inputenc}
\usepackage{comment}
\usepackage{amssymb,amsfonts}
\usepackage[style=alphabetic, backend=biber, sorting=nyt, doi=false, isbn=false,url=true, hyperref,backref,backrefstyle=none, maxnames=20, maxalphanames=20]{biblatex}
\usepackage{mathtools}
\usepackage{nicematrix}
\usepackage{enumerate}
\usepackage{mathrsfs}
\usepackage{mathtools}
\usepackage{xcolor}
\usepackage[normalem]{ulem}
\usepackage{hyperref}
\usepackage{thm-restate}
\hypersetup{
    colorlinks=true, % This is important to make the links themselves colored, not just a colored box around them
    citecolor=blue,  % Change 'blue' to your desired color (e.g., 'red', 'green', 'Violet')
    linkcolor=red,   % Optional: Sets the color for internal links (e.g., table of contents)
    urlcolor=cyan    % Optional: Sets the color for URLs
}
\usepackage{cleveref}
\usepackage{blkarray}
\usepackage{amsthm}
\theoremstyle{definition}
\newtheorem{thm}{Theorem}[section]
\renewcommand{\thethm}{%
  \ifnum\value{section}>0
    \thesection.\arabic{thm}%
  \else
    \arabic{thm}%
  \fi
}
\theoremstyle{definition}

\newtheorem{cor}[thm]{Corollary}
\newtheorem{prop}[thm]{Proposition}
\newtheorem{lem}[thm]{Lemma}

\newtheorem{defn}[thm]{Definition}

\newtheorem{rem}[thm]{Remark}
\newtheorem*{rems}{Remark}

\def\C{{\mathbb C}}
\def\Z{{\mathbb Z}}

\def\Q{{\mathbb Q}}

\def\C{{\mathbb C}}

\def\Qbar{\overline{\mathbb{Q}}}

\DeclareMathOperator{\Spec}{Spec}

\DeclareMathOperator{\univ}{univ}
\DeclareMathOperator{\proj}{Proj}

\DeclareMathOperator{\divi}{div}
\DeclareMathOperator{\Jac}{Jac}

\DeclareMathOperator{\Pic}{Pic}

\DeclareMathOperator{\tar}{tar}

\DeclareMathOperator{\trdeg}{trdeg}

\DeclareMathOperator{\ord}{ord}

\DeclareMathOperator{\tors}{tors}

\DeclareMathOperator{\amp}{amp}

\usepackage{tikz-cd}
\newcommand{\CH}{\operatorname{CH}}

\newcommand{\Cer}{\operatorname{Cer}}

\usepackage{tikz-cd} 

\hypersetup{ colorlinks=true,
linkcolor=blue
}
\begin{document}
\title{Covers of curves, Ceresa cycles, \& Unlikely intersections}
\author[T. Bhatnagar]{Tejasi Bhatnagar}\address{The Ohio State University, 100 Math Tower, 231 W 18th Ave, Columbus, OH 43210}
\email{\href{mailto:bhatnagar.62@osu.edu}{bhatnagar.62@osu.edu}}
\urladdr{\url{https://sites.google.com/view/tbhatnagar/home?authuser=0}}
\email{}
\author[S. Devadas]{Sheela Devadas}
\address{University of Maine, Neville Hall, Room 237, 
Orono, ME 04469}
\email{\href{mailto:sheela.devadas@maine.edu}{sheela.devadas@maine.edu}}
\urladdr{\url{https://sheeladevadas.github.io}}
\author[T. D'Nelly-Warady]{Toren D'Nelly-Warady}
\address{University of California, San Diego, 9500 Gilman Drive
La Jolla, CA  92093, USA}
\email{\href{mailto:tdnellywarady@ucsed.edu}{tdnellywarady@ucsd.edu}}

\author[P.~Srinivasan]{Padmavathi Srinivasan}
\address{Boston University, 665 Commonwealth Avenue, Boston, MA 02215, USA}
\email{\href{mailto:padmask@bu.edu}{padmask@bu.edu}}
\urladdr{\url{https://padmask.github.io/}}

\date{}

\begin{abstract} Fix a smooth, projective, geometrically integral curve $C$ of genus $g \geq 2$ over a characteristic zero field. We prove that the Ceresa cycle $\Cer(\widetilde{C})$ of a very general ramified cover $\widetilde{C}$  of $C$ is nontorsion in the Chow group of its Jacobian. We also show that there exist infinitely many families of ramified covers of a varying family of curves where a general point of these families corresponds to a curve with nontorsion Ceresa cycle. To illustrate this, we write down two explicit $1$-dimensional and $2$-dimensional families of genus $6$ curves where the locus of curves with torsion Ceresa cycle is Zariski closed and has positive codimension. Our strategy is to reduce the question of whether the Ceresa cycle is torsion to the question of whether a related point on the Jacobian of the curve is torsion. For this, we use the  ``relative canonical shadow" of the Ceresa cycle, which is a point in the Jacobian of the curve obtained by intersecting the Ceresa cycle with a natural correspondence arising from the covering map. We combine this with ideas from unlikely intersection theory (namely the relative Manin--Mumford theorem) to study the locus where the relative canonical shadows of the Ceresa cycle become torsion.
\end{abstract}
\maketitle

Let $K$ be a characteristic zero field. Let $C$ be a nice (i.e. smooth, projective, geometrically integral) curve over $K$ of genus $g \geq 2$. Let $J$ be the Jacobian of $C$. Let $e$ be any degree $1$ divisor on $C$, and let $i_{e} \colon C \rightarrow J$ denote the associated Abel-Jacobi embedding. The Ceresa cycle is the cycle  \[ \Cer_e(C) \colonequals [i_e(C)] - [-1]^*[i_e(C)] \in \CH_1^{\hom}(J) \]
in the Chow group of homologically trivial $1$-cycles on $J$ modulo rational equivalence. If $\Cer_e(C)$ is torsion in $\CH_1(J)$, then $(2g-2)e$ is a canonical divisor of $C$ in $\CH_0(C)$  \cite[Remark~3]{ELS}. Furthermore, if $\Cer_e(C)$ is torsion for one such choice of $e$, then it is torsion for all such $e$  \cite[Lemma~2.11]{AriJef2}. Henceforth we fix such a divisor $e$ and drop the subscript $e$ in $\Cer_e(C)$. 

In \cite{Ceresa}, Ceresa showed that $\Cer(C)$ is nontorsion modulo algebraic equivalence for a very general complex curve $C$ when $g \geq 3$. Such homologically trivial nontorsion cycles are hard to construct in higher codimensions. As a result, the Ceresa cycle has played an important role in the study of Chow groups and Griffiths groups and in the formulation of the Beilinson--Bloch conjectures \cite{Beilinson}, \cite{Bloch1}, \cite{Bloch2}. We prove the following analog of Ceresa's result. 

Let $d \geq 1$ be a positive integer, let $G$ be a transitive subgroup of $S_d$, let  $\Sigma_g$ be the genus $g$ oriented surface, let $B \subset \Sigma_g$ be a finite subset of size $r$, and let $\rho$ be a conjugacy class of homomorphisms from the fundamental group $ \pi_1(\Sigma_g \setminus B)$ to $G$.  Let $\mathcal{H}_{d,r}(g,\rho)$ (and $\mathcal{H}_{d,r}(C,\rho)$ respectively) be the Hurwitz space of degree $d$ covers of smooth genus $g$ curves (of the specific genus $g$ curve $C$ respectively) whose monodromy representation is in the conjugacy class of the representation $\rho$
(see \cite[Theorem~4]{Wewers} for the construction of $\mathcal{H}_{d,r}(g,\rho)$ and $\mathcal{H}_{d,r}(C,\rho)$).

\begin{thm}\label{thm:Galcoverofggeq2} 
Assume $g \geq 2$ and that $\rho$ corresponds to a ramified cover of $C$. Then $\Cer(\widetilde{C})$ is nontorsion in the Chow group for a very general point $[\widetilde{C} \rightarrow C]$ in $ \mathcal{H}_{d,r}(C,\rho)(\C)$. 
\end{thm}
\begin{rems}
A curve that dominates a curve with nontorsion Ceresa cycle has nontorsion Ceresa cycle \cite[Proposition~25]{CerMod}, and so if $\Cer(C)$ is nontorsion, then so is $\Cer(\widetilde{C})$ for every  $[\widetilde{C} \rightarrow C] \in \mathcal{H}_{d,r}(C,\rho)$. Nevertheless, Theorem~\ref{thm:Galcoverofggeq2} gives us the weaker conclusion that $\Cer(\widetilde{C})$ is nontorsion for a  very general $[\widetilde{C} \rightarrow C] \in \mathcal{H}_{d,r}(C,\rho)$ even when we do not know if $\Cer(C)$ itself is torsion.  Theorem~\ref{thm:Galcoverofggeq2} is also interesting when $\Cer(C)$ is torsion, such as when $C$ is hyperelliptic.
\end{rems}

\begin{rems}
    In \cite[Theorem~4.4.1]{ColPir}, Collino and Pirola extended Ceresa's result to show that $\Cer(C)$ is nontorsion for a very general point of any subvariety $\mathcal{V}$ of codimension $ < (g+2)/3$ in the moduli space $\mathcal{M}_g$ of genus $g$ curves when $g > 3$. Theorem~\ref{thm:Galcoverofggeq2} produces new examples of subvarieties (whose very general point corresponds to a curve with nontorsion Ceresa cycle) not automatically covered by Collino and Pirola's result.
    %, since the image of the locus $\mathcal{H}_{C,r}(\rho)$ of Theorem~\ref{thm:Galcoverofggeq2} under the natural source curve map to  $\mathcal{M}_{\tilde{g}}$ has codimension $3(\tilde{g}-g_{C})-r$, which can be larger than $(\tilde{g}+2)/3$. 
\end{rems}

We also show that there exist infinitely many families of ramified covers of a {\textit{varying}} family of curves where a {\textit{general}} point of these families corresponds to a curve with nontorsion Ceresa cycle. Let $\mathcal{M}_g$ and $\mathcal{M}_{g,r}$ be fine moduli spaces of nice unpointed and $r$-pointed genus $g$ curves with  full symplectic level-$5$ structure respectively. Let $\tar \colon \mathcal{H}_{d,r}(g,\rho) \rightarrow \mathcal{M}_{g,r}$ denote the target map sending a cover to its target marking the branch points. For a locally closed subvariety $Z$ of $\mathcal{M}_g$, let $Z_r \colonequals \mathcal{M}_{g,r} \times_{\mathcal{M}_g} Z$ and $f_Z \colon Z_r \rightarrow Z$ denote the natural projection map. 

\begin{thm}\label{T:CerMMgen} Assume $g \geq 2$ and that $\rho$ corresponds to a ramified cover. Let $Z$ be a locally closed integral subvariety of $\mathcal{M}_g$ defined over $\Qbar$, and let $Z_r,f_Z$ be as above. Assume $\dim Z < g$. Then as we vary over nonempty \'{e}tale open subsets $U$ of $Z$, there are infinitely many sections $s \colon U \rightarrow ({Z_r})_U$ of $f_U$ such that if  $\mathcal{H}$ is an irreducible component of $\tar^{-1}(s(U))$, 
then  $\Cer(\widetilde{C})$ is nontorsion for $[\widetilde{C}\rightarrow C]$ in an open dense subset of $\mathcal{H}(\C)$. 
\end{thm}

To illustrate \Cref{T:CerMMgen}, we produce an explicit $1$-dimensional family and a $2$-dimensional family of genus $6$ curves that are double covers of a varying family of genus $3$ curves where the locus  of curves with nontorsion Ceresa cycle is open and dense.  
\begin{thm}\label{T:CerMM2}
 Consider the $1$-parameter family of genus $6$ curves $C_t$ with affine equation 
 \[ y^{12} = \left( \frac{x+1}{x-1} \right)^3 \left( \frac{x+t}{x-t} \right)^4. \]
 Then $\Cer(C_t)$ is nontorsion for every $t \in \C \setminus \{0,\pm 1\}$.
\end{thm}
The family above was studied by Laga and Shnidman in \cite[Example~3.9]{AriJef2}, where they proved that $\Cer(C_t)$ is torsion in the Griffiths group for every $t$, and remark that it would be interesting to understand the locus of $t$ where $\Cer(C_t)$ becomes torsion in the Chow group. 

Let $B \subset \Sigma_3$ be a $2$-element subset of a genus $3$ oriented surface. Let $\rho \colon \pi_1(\Sigma_3 \setminus B) \rightarrow \Z/2\Z$ be any homomorphism that sends the generating loops around the elements of $B$ to the nontrivial element of $\Z/2\Z$, and let $\mathcal{H}_{2,2}(3)$ denote the corresponding Hurwitz space. 

 \newcommand{\thmcermmpream}{Consider the $2$-parameter family of plane quartic genus $3$ curves cut out by 
\[ f_{u,w}(X,Y,Z) \colonequals w^2(X(Y^3+Z^3)+Y^2Z^2)+(u^3+w^4)X^2YZ+w^2u^3X^4=0. \]
Let $A \colonequals [1:w:-w], A' \colonequals [1:-u:0]$ be sections of the projection $\pi_2 \colon V(f) \rightarrow \mathbb{A}^2_{(u,w)}$. Let $U$ be the image of the smooth locus of $\pi_2$, let $U \rightarrow \mathcal{M}_{3,2}$ be the moduli map corresponding to $(\pi_2,A,A')$, and let $S \subset \mathcal{M}_{3,2}$ be its image.  Let $\mathcal{H}$ be an irreducible component of $\tar^{-1}(S)$. }

\begin{restatable}{thm}{TCerMMone} \label{T:CerMM1} 
   \thmcermmpream %Then $\Cer(\widetilde{C})$ is nontorsion for an open dense subset of $t \in U(\C)$. 
   Then $\Cer(\widetilde{C})$ is nontorsion for $[\widetilde{C} \rightarrow C]$ in an open dense subset of $\mathcal{H}(\C)$.
\end{restatable}

Our results fit inside the recent body of results towards understanding the locus $\mathcal{M}_{g}^{\tors}$ inside $\mathcal{M}_g$ consisting of curves with torsion Ceresa cycle. Surprisingly little is known about $\mathcal{M}_{g}^{\tors}$, besides the fact that it has at most countably many irreducible components, and that $\mathcal{M}_{g}^{\tors}$ contains the hyperelliptic locus. (See \cite{AriJef2}, \cite{QiuZhang}, \cite{Laterveer} for some other Hurwitz strata contained in $\mathcal{M}_g^{\tors}$ for some small $g$). Ceresa's original result shows that $\mathcal{M}_{g}^{\tors} \neq \mathcal{M}_g$ when $g \geq 3$, and recent results of  Gao and Zhang \cite{GaoZhang}, Hain \cite{Hain}, Kerr and Tayou \cite{KerrTayou} show that there is an open dense subvariety $\mathcal{M}_g^{\amp}$ of the coarse moduli space $\mathcal{M}_g$ of genus $g$ curves where the Ceresa cycle is torsion for at most {\textit{countably}} many\footnote{Gao and Zhang prove that all curves in $\mathcal{M}_g^{\amp}$ with torsion Ceresa cycle are defined over $\overline{\mathbb{\Q}}$ and that the Beilinson--Bloch heights of Ceresa cycles in $\mathcal{M}_g^{\amp}(\overline{\Q})$ satisfy the Northcott property, i.e., there are finitely many points of $\mathcal{M}_g^{\amp}(\overline{\Q})$ of bounded degree and whose Ceresa cycles have bounded height.} points of $\mathcal{M}_g^{\amp}(\C)$ when $g \geq 3$.  
The locus $\mathcal{M}_g^{\tors}$ is contained in the torsion locus for the Abel--Jacobi image of the Ceresa cycle, and Gao--Zhang, Hain, Kerr--Tayou use different Hodge--theoretic techniques to analyze the torsion locus for the Abel--Jacobi image of the Ceresa cycle. We do not know if the generic points of the Hurwitz spaces that we consider map to the open subset $\mathcal{M}_g^{\amp}$ in $\mathcal{M}_g$, or if the techniques in these papers can be used to produce families where the Ceresa cycle is nontorsion on an open dense locus as in our \Cref{T:CerMMgen}, \Cref{T:CerMM2} and \Cref{T:CerMM1}. To our knowledge, these theorems give some of the first examples where the Ceresa cycle is known to be generically nontorsion in the Chow group and where the torsion locus is provably Zariski closed. 

\subsection*{Overview of proof strategy} 
Our results are inspired by the recent results of Laga and Shnidman \cite[Theorem~C]{AriJef2}, where the authors show that the family of genus $3$ Picard curves has infinitely many $\overline{\mathbb{Q}}$-specializations that have torsion Ceresa cycle. Laga and Shnidman use the theory of Chow motives together with known cases of the Hodge conjecture to show that the Ceresa cycle in the family of genus $3$ Picard curves becomes torsion precisely when a section on a closely related elliptic fibration becomes torsion. We instead use intersection theory of algebraic cycles to reduce the question of whether the Ceresa cycle is torsion to a question of whether a closely related point on the Jacobian of the curve is torsion. To wit, for curves arising from covering constructions as in our theorems, intersecting the Ceresa cycle with a natural correspondence arising from the covering map produces a point on the Jacobian of the curve that is torsion when the Ceresa cycle is torsion. This point is called the ``relative canonical shadow" of the cover (Definition~\ref{D:RCShadow}) and was introduced by Ellenberg, Logan and the fourth named author in \cite{ELS} as an obstruction to triviality of the Ceresa cycle. The main new idea in this paper is to study this obstruction in families using results in unlikely intersection theory. 

More precisely, we use the Manin--Mumford--Raynaud theorem (\Cref{thm:manin-mumford}) to show that the relative canonical shadow is nontorsion for very general choices of branch points for a ramified  cover of $C$ when the genus of $C$ is at least $2$. In \Cref{T:CerMM2} and Theorem~\ref{T:CerMM1}, we study double covers that admit many {\textit{independent}} maps to elliptic curves. The family of curves in \Cref{T:CerMM2} admits a $D_{12}$ action. For proving \Cref{T:CerMM2}, we push forward the relative canonical shadow for quotients by two of the involutions to produce two sections of an elliptic fibration. If the specialization of the Ceresa cycle is torsion, then both these sections have to \textit{simultaneously} specialize to torsion points. We prove that the sections are linearly independent by specializing to a particular member in the family. We then appeal to the Masser--Zannier unlikely intersections theorem (\Cref{thm:RelManMum} and \Cref{rem:MasserZannier}) to argue that there are at most finitely points $t \in \C$ where both sections specialize to torsion points. We then use a result of Stoll \cite[Corollary~22]{Stoll} to check that this simultaneous torsion locus is actually empty. 

For Theorem~\ref{T:CerMM1}, we similarly use the $S_3$ action on the smooth plane quartics combined with a specialization argument to prove that the abelian subvariety generated by the relative canonical shadow in the family of relative Jacobians has dimension at least $3$ (\Cref{cor:ZarClosDimension}).  Since $3 > 2 = \dim \mathcal{H}$, we can apply the recent relative Manin--Mumford theorem proved by Gao and Habegger (Theorem~\ref{thm:RelManMum}) (a common generalization of the Manin--Mumford--Raynaud and Masser--Zannier results above) to conclude that the torsion locus is Zariski closed and has positive codimension in $\mathcal{H}$ (\Cref{T:CerMM1}). A similar strategy works for \Cref{T:CerMMgen}, with one key additional input provided by Habegger's ``bounded height theorem" (\Cref{thm:Habeggerboundedheight}) which implies that for a nice curve $C$ defined over $\Qbar$ embedded in its Jacobian $J$, there are infinitely many generating points of $C(\Qbar)$, that is, points $P$ such that the smallest abelian subvariety of $J$ containing $P$ is $J$ itself. This allows us to verify the necessary dimension inequality needed for applying the relative Manin--Mumford theorem. 

\subsection*{Special subvarieties of $\mathcal{M}_g$} 
In the spirit of Andr\'{e}--Oort style conjectures on Shimura varieties, one might hope that for a subvariety $X$ of $\mathcal{M}_g$, the intersection $X(\mathbb{C}) \cap {\mathcal{M}}_{g}^{\tors}$ is not Zariski dense in $X$ unless $X$ itself is special in some way.\footnote{Having a Zariski dense subset of Ceresa-torsion points is one of several competing notions of special subvarieties of $\mathcal{M}_g$, see \cite{Baldi} for other notions.} Baldi, Klingler, and Ullmo \cite[Conjecture~5.12]{BKU}, generalizing earlier work of Klingler \cite[Section~1.5]{Klingler}, have recently formulated a Zilber--Pink type conjecture for variations of  mixed Hodge structures, and this framework allows one to characterize  special subvarieties of $\mathcal{M}_g$ using a natural variation of mixed Hodge structures on $\mathcal{M}_g$ arising from the normal function of the Ceresa cycle (see \cite[Section~2]{KerrTayou}). The families in \Cref{T:CerMMgen}, Theorem~\ref{T:CerMM2} and Theorem~\ref{T:CerMM1} where the torsion locus is contained in a  Zariski closed subvariety of positive codimension (and the family in \cite[Theorem~D]{AriJef2}\footnote{See \cite[Example~4.3]{KerrTayou} for more about the connection of this family to the locus $\mathcal{M}_g^{\amp}$ mentioned above.} with a Zariski dense Ceresa torsion locus respectively) provide a natural setting to test this generalized Zilber--Pink style conjecture, if one is able to verify that the families that we study are indeed non-special (respectively special) in this generalized setting.

\section*{Acknowledgements}
This project started at the Women in Algebraic Geometry (WIAG) workshop held at the Institute of Advanced Study in 2024. We are grateful to the organisers of WIAG for giving us the opportunity to be a part of this wonderful workshop. We thank Jerson Caro,  Jordan Ellenberg, Jef Laga, Rachel Pries, Ari Shnidman, and Salim Tayou for helpful conversations. We thank ICERM for supporting the conference ``The Ceresa cycle in arithmetic and geometry" which  inspired this paper. T.B. was supported by NSF MSPRF 2503371 in 2025-2026. P.S. was supported by National Science Foundation grant DMS-2401547.

\subsection*{Notation}\label{sec:background}
By a general (resp. very general) point of an integral variety $V$, we mean a geometric point lying outside of a proper closed subvariety $W$ of $V$ (resp. a countable union of such $W$). 
A nice variety $X$ over a field $K$ is a smooth, projective, geometrically integral $K$-variety. For a nice curve $C$, we let $g(C)$ denote the genus of $C$ and $\Jac(C)$ denote the Jacobian of $C$. By a subgroup variety of an abelian variety $A$, we mean a closed and not necessarily irreducible subvariety of $A$ that is closed under addition and negation. For an abelian variety $A$ over a characteristic zero field $K$, we let $A_{\tors} \subset A(\C)$ denote the set of all geometric torsion points of $A$. For an abelian scheme $\mathcal{A} \rightarrow S$, we let $\mathcal{A}_{\tors}$ denote the union of ${\mathcal{A}_s}_{\tors}$ over all geometric points $s$ of $S$. For an abelian scheme $A$ and an integer $n$, we let $[n]$ denote the multiplication by $n$ map on $A$. 
For a point $a$ on an abelian variety $A$ and a subvariety $X$ of $A$, we let $a+X$ denote the translate of $X$ by $a$. 
For a subvariety $X$ of an abelian $S$-scheme $\mathcal{A}$, we let $\Z X$ denote the countable union $\cup_{N \in \Z} [N](X)$ in $\mathcal{A}$ and $\overline{\Z X}$ the Zariski closure of $\Z X$ in $\mathcal{A}$. 

\section{Relative canonical shadows of Ceresa cycles for separable covers}
 
In this section we let $C$ and $C'$ denote nice curves over a field $K$ with genera $g_C,g_{C'}$ respectively. Let $\varphi \colon C \rightarrow C'$ be a separable cover of curves of degree $d$ and ramification divisor $R_{\varphi}$. Let $K_{C'}$ be a canonical divisor for $C'$.
\begin{defn}\label{D:RCShadow} \cite[Definition~5.1]{ELS}  
 The relative canonical shadow of $\varphi$ is the class in $\Pic^0(C)$ of the divisor 
\[ D_{\varphi} \colonequals d (2g_{C'}-2)R_{\varphi}-\deg(R_{\varphi}) \varphi^{*}(K_{C'}) + 2(dR_{\varphi}-\varphi^*\varphi_*R_{\varphi}).\]
\end{defn}

\begin{rem}
The definition of the relative canonical shadow of a finite cover (not necessarily separable) in \cite[Definition~5.1]{ELS} is 
\begin{equation}\label{eqn:ELSformula} D_{\varphi} = -(2g_C - 2)\varphi^*(K_{C'}) - 2\varphi^*\varphi_*(K_C) + (2dg_{C'})K_C \end{equation}
When $\varphi$ is separable, using the Riemann-Hurwitz formula to rewrite $K_C$ as $\varphi^*(K_{C'})+{R_{\varphi}}$ and $2g_C-2$ as $d(2g_{C'}-2) + \deg(R_{\varphi})$ in $D_{\varphi}$, together with the fact that $\varphi_* \varphi^*$ is multiplication by $d$ on divisors on $C'$, we get the expression in Definition~\ref{D:RCShadow}.
\end{rem}

Our interest in the relative canonical shadow is due to the following proposition.
\begin{prop}\label{prop:rcsceresa}\cite[Lemma~2.3, Remark~9, Lemma~5.2]{ELS}
If $D_{\varphi}$ is infinite order, then the Ceresa cycle $\Cer_e(C)$ of $C$ has infinite order for every degree $1$ divisor $e$ on $C$.
\end{prop}

\begin{rem}\label{rem:RCspecial} (See also \cite[Remark~18]{ELS}) We record a few special cases where the expression for the relative canonical shadow simplifies further.
\begin{enumerate}
\item If $\varphi$ is \'{e}tale, then $K_C = \varphi^*(K_{C'}), \varphi^*\varphi_*(K_C) = \deg(\varphi) K_C$, and so $D_{\varphi} = 0$.
\item If $g_{C'} = 0$, then $D_{\varphi} = \varphi^* \left( -(2g_C - 2)K_{C'} - 2\varphi_*(K_C) \right) = 0$ since $\Pic^0(C') = 0$.
\item\label{rem:gE1} If $g_{C'}=1$, then $K_{C'}$ is trivial, and
$ D_{\varphi} = 2dR_{\varphi}-2\varphi^*\varphi_*R_{\varphi}$. 
\item\label{rem:Gal} If $\varphi$ is a $G$-Galois cover, then $\varphi^*\varphi_*(R_{\varphi}) = \sum_{\sigma \in G} \sigma(R_{\varphi})=dR_{\varphi}$ and therefore \[D_{\varphi} = d (2g_{C'}-2)R_{\varphi}-\deg(R_{\varphi}) \varphi^{*}(K_{C'}) = (2g_C-2)R_{\varphi} - \deg(R_{\varphi}) K_C.\] 
\item\label{rem:GalE} If $g_{C'} = 1$ and $\varphi$ is Galois, then $K_{C'}$ is trivial and by the previous part $D_{\varphi} = 0$.
\end{enumerate}
So the situations where $D_{\varphi}$ is not immediately zero are covers of curves of genus at least $2$  
and non-Galois covers of curves of genus $1$ curves. 
\end{rem}

 We now prove a lemma about the push-forward of the  relative canonical shadow.
 \begin{lem}\label{lem:relcanpush}
The divisor $\varphi_*(D_{\varphi}) = d ((2g_{C'}-2) \varphi_*(R_{\varphi})-\deg(R_{\varphi}) K_{C'}).$ In particular, if $g_{C'} = 1$, then  $\varphi_*(D_{\varphi}) = 0$.
\end{lem}
\begin{proof} 
Since $\varphi_* \varphi^*$ is multiplication by $d$ on $\Pic^0(C')$, it follows that
\begin{align}\label{E:ShPush}
\begin{split}
\varphi_*(D_{\varphi}) &= \varphi_*(d (2g_{C'}-2)R_{\varphi}-\deg(R_{\varphi}) \varphi^{*}(K_{C'}) + 2(dR_{\varphi}-\varphi^{*}(\varphi_*)R_{\varphi})) \\
&=d (2g_{C'}-2) \varphi_*(R_{\varphi})-\deg(R_{\varphi}) \varphi_*(\varphi^{*}(K_{C'})) + 2(d \varphi_*(R_{\varphi})-\varphi_*(\varphi^{*}(\varphi_*)R_{\varphi})) \\
&= d (2g_{C'}-2) \varphi_*(R_{\varphi})-d\deg(R_{\varphi}) K_{C'} + 2d \varphi_*(R_{\varphi})-2d\varphi_*(R_{\varphi}) \\
&= d ((2g_{C'}-2) \varphi_*(R_{\varphi})-\deg(R_{\varphi}) K_{C'}).\qedhere
\end{split}
\end{align}
\end{proof}

\section{Unlikely intersections in abelian varieties}
In this section, we collect some results in unlikely intersections that we will use to prove the main theorems of this paper in the next two sections.
 
\begin{thm}\label{thm:manin-mumford}(Manin--Mumford--Raynaud, \cite[Theorem~3.1]{MasserZannier})
Let $X$ be a subvariety of an abelian variety $A$ defined over a field of characteristic zero. The torsion points in $X$ are contained and Zariski dense in a finite union of torsion translates of abelian subvarieties of $A$ contained in $X$. In particular,  if $\dim X = 1$ and $X$ contains infinitely many torsion points, then $X$ is a torsion translate of an elliptic curve.
\end{thm}

Recall the definitions of $\Z X$ and $\overline{\Z X}$ from the notation section. 

\begin{thm}\label{thm:RelManMum}(The Relative Manin--Mumford conjecture \cite[Theorem~1.1]{RelManMum})
Let $S$ be a nice variety over an algebraically closed field $L$ of characteristic $0$. Let $\pi \colon \mathcal{A} \rightarrow S$ be an abelian scheme of relative dimension $g \geq 1$. Let $X$ be an irreducible subvariety of $\mathcal{A}$. Assume that $\overline{\Z X} = \mathcal{A}$. If $X(L) \cap \mathcal{A}_{\tors}$ is Zariski dense in $X$, then $\dim X \geq g$.
\end{thm}

\begin{rem}\label{rem:MasserZannier} When $S$ is a point, the relative Manin--Mumford conjecture reduces to the Manin--Mumford--Raynaud theorem above. When $\mathcal{A} = \mathcal{E}^2 \rightarrow S$  for an elliptic $S$-scheme $\mathcal{E}$ with $\dim S=1$, the relative Manin--Mumford conjecture reduces to the Masser--Zannier unlikely intersections theorem \cite{MasserZanniertorsion},
by taking $X$ to be the image of a section $\tau\times\sigma \colon S\rightarrow \mathcal{E}^2$ coming from a pair of linearly independent sections of $\mathcal{E}$ (note that $\overline{\Z X} = \mathcal{E}^2$ exactly when the sections $\tau$ and $\sigma$ are linearly independent).
In this case, Theorem \ref{thm:RelManMum} implies that there are at most finitely many points $t\in S$, such that $\tau(t)$ and $\sigma(t)$ are simultaneously torsion. In \Cref{sec:one-parameter}, we apply Theorem \ref{thm:RelManMum} with $\dim S = 1$ and $\dim S = 2$ to abelian $S$-schemes that are products of elliptic $S$-schemes to prove Theorems \ref{T:CerMM2} and  \ref{T:CerMM1}. 
\end{rem}

\begin{defn} Let $Z$ be a subvariety of an abelian variety $A$. We say that $Z$  generates $A$ if the smallest abelian subvariety of $A$ containing $Z$ is $A$.  
\end{defn}

Let $A$ be an abelian variety defined over $\Qbar$. Fix a N\'{e}ron--Tate height on $A(\Qbar)$ associated to some symmetric and ample line bundle. Let  $X \subset A$ be a subvariety also defined over $\overline{\Q}$. We now state a result of Habegger about points in $X(\Qbar)$ lying outside all positive codimension abelian subvarieties in the special case when $\dim X = 1$. By the discussion at the bottom of \cite[Page~406]{HabeggerCompDim}, when $\dim X = 1$, the subset $X^{oa,[1]} \subset X$ in \cite[Page~407, Theorem]{HabeggerCompDim} is equal to $X$ when $X$ is not contained in any translate of an abelian subvariety of $A$ of codimension $\geq 1$.
\begin{thm}\label{thm:Habeggerboundedheight}(\cite[Theorem]{HabeggerCompDim}) Let $X \subset A$ be a subvariety of an abelian variety, both defined over $\Qbar$. Let $T$ be the set of all abelian subvarieties of $A$ of codimension $\geq 1$. Assume $\dim X = 1$ and that $X$ is not contained in a translate of any element of $T$. Then the set of points in $X(\Qbar) \cap \left(\bigcup_{B \in T} B \right)$ has bounded height. In particular, there are infinitely many points in $X(\Qbar)$ that generate $A$.
\end{thm}

\subsection{Generating subvarieties of abelian varieties}

In this section, we produce examples of subvarieties $X$ of an abelian scheme $\mathcal{A} \rightarrow S$ that satisfy the hypothesis $\overline{\Z X} = \mathcal{A}$ of Theorem~\ref{thm:RelManMum}.

We begin with a lemma that will be used in the proof of Theorem~\ref{thm:Galcoverofggeq2}.

Let $C$ be a nice curve over a field $K$, let $A$ be an abelian variety and let $i \colon C \rightarrow A$ be a $K$-morphism. Let $r \geq 1$ be an integer, and let $\vec{a} \colonequals (a_1,a_2,\ldots,a_r)$ be an $r$-tuple of positive integers.  Let $i_r \colonequals i_r^{\vec{a}} \colon C^r \rightarrow A$ be the map defined by $i_r(x_1,\ldots,x_r) \colonequals \sum_{i=1}^r a_i i(x_i)$.

\begin{lem}\label{lem:curvesinJacobians} Let $C,A,i,i_r$ be as above. Assume that $i(C)$ generates $A$. Let $X \colonequals i_r(C^r)$. Then either $X = A$, or $\dim X < \dim A$ and any torsion translate of an abelian subvariety contained in $X$ has positive codimension in $X$.
\end{lem}

\begin{proof}
If $\dim X \geq \dim A$, then $X=A$ since $A$ is irreducible and $X$ is a closed subvariety of $A$. So now assume that $\dim X < \dim A$. 

Let $a \in A$ be a torsion point of order $n$, and let $\tilde{B}$ be an abelian subvariety of $A$ such that the torsion translate $T \colonequals a+\tilde{B}$ is contained in $X$. Let $j_r \colonequals [n] \circ i_r$, let $Y \colonequals [n](X) = j_r(C^r)$ and let $B \colonequals [n](T)$, so $B$ is a closed subvariety of $Y$. Since $[n]$ is an isogeny and $\ord(a) = n$, it follows that $B = [n](a+\tilde{B}) = [n](\tilde{B})$ is an abelian subvariety of $A$. As $[n]$ is a finite map, we have $\dim B = \dim T$ and $\dim Y = \dim X$, so it suffices to show $\dim B < \dim Y$.

Suppose that $\dim B \geq \dim Y$. Then $B=Y$ since $B$ is a closed subvariety of $Y$ and $Y = j_r(C^r)$ is irreducible. We now argue that if $Y$ is an abelian variety then $Y=A$, which contradicts our assumption that $\dim Y = \dim X < \dim A$. Let $d \colonequals \sum na_i$. Pre-composing $j_r$ with the diagonal embedding $ C \rightarrow C^r$ shows that $Y$ contains $[d]_*(i(C))$. Since $i(C)$ generates $A$ by assumption, it follows that $[d]_*(i(C))$ generates $[d]_*(A) = A$.  Therefore $Y=A$.
\end{proof}

For the rest of this section, let $C$ be a nice curve with Jacobian $J$. For a degree $1$ divisor $D$ on $C$, let $i_D \colon C \rightarrow J$ denote the corresponding Abel--Jacobi map defined by $i_D(P) \colonequals [P-D]$. 
\begin{lem}\label{lem:abeliantranslate}
Let $C$ be a nice curve. Assume $g(C)\geq 2$. Fix a degree $1$ divisor $P$ and a point $a  \in J$. Fix a nonzero $n\in \Z$. Then $a+[n](i_{P}(C))$ is not contained in a translate of an abelian subvariety of $J$ of positive codimension.
\end{lem}
\begin{proof} 
Assume $a+[n](i_{P}(C)) \subset a'+B$ for some $a' \in J$ and abelian subvariety $B$ of $J$ of positive codimension. Then $(a-a')+[n](i_{P}(C)) \subset B$. Replacing $a$ by $a-a'$, we may assume that $a+[n](i_{P}(C)) \subset B$ without loss of generality. Let $Q$ be a degree $1$ divisor such that $a+[nP-nQ] = 0 \in J$. Then
\[ [n](i_{Q}(C)) = [n]((Q-P)+i_{P}(C)) = [n(Q-P)]+[n](i_{P}(C)) = a+[n](i_{P}(C)) \subset B.\]
If $[n](i_{Q}(C)) \subset B$, then $i_{Q}(C) \subset [n]^{-1}(B)$. Since $n \neq 0$, it follows that $[n] $ is an isogeny and hence $[n]^{-1}(B)$ is a finite union of translates of abelian subvarieties of $J$ of positive codimension. Since $i_{Q}(C)$ is irreducible, it has to be contained in one of them, and once again by adding an appropriate degree $0$ divisor to the the degree $1$ divisor $Q$ we may assume that $i_{Q}(C)$ is contained in an abelian variety of positive codimension in $J$ without any loss of generality. This is a contradiction since $i_Q(C)$ generates $J$ by our assumption that $g \geq 2$.
\end{proof}

\begin{lem}\label{lem:infintepoints}Fix an integer $r \geq 1$, an $r$-tuple of positive integers $(a_1,\ldots,a_r)$, and let $C,i,i_r$ be as in \Cref{lem:curvesinJacobians}. Assume $C$ is defined over $\Qbar$ and $g(C) \geq 2$. Fix pairwise distinct points $P_2,\dots P_r$ of $C(\Qbar)$. There exist infinitely many $P \in C(\Qbar)$ such that $\overline{\Z i_r(P,P_2,\ldots,P_r)} = J$.  
\end{lem}
\begin{proof}
Let $p \colon C\rightarrow C^r$ be defined by $p(P) \colonequals (P,P_2,\dots, P_r)$ and let $j_r \colonequals i_r \circ p$. Let $a \colonequals \sum_{i=2}^r a_i i(P_i) \in J$. Then $j_r(C) = a + [a_1](i(C))$, and since $a_1 \neq 0$, it follows that $j_r(C) \subset J$ is a curve defined over $\Qbar$. \Cref{lem:abeliantranslate} implies that $j_r(C)$ is not contained in a translate of an abelian subvariety of $J$ of positive codimension. Applying \Cref{thm:Habeggerboundedheight} to $j_r(C) \subset J$, we conclude that there exist infinitely many $P \in C(\Qbar)$ such that $J = \overline{\Z j_r(P)} = \overline{\Z i_r(P,P_2,\ldots,P_r)}.$\end{proof}

\section{A very general cover of a  hyperbolic curve has nontorsion Ceresa cycle}\label{sec:proofmainthm}

\begin{proof}[Proof of Theorem~\ref{thm:Galcoverofggeq2}]
Let $[\varphi \colon \widetilde{C} \rightarrow C] \in \mathcal{H}_{C,r}(\rho)$. Let $d, R_{\varphi},D_{\varphi}$ be as in Lemma~\ref{lem:relcanpush}.  %Let $\deg(\varphi) = d$. Let $R_{\varphi}$ be the ramification divisor of $\varphi$ and let $D_{\varphi}$ be the relative canonical shadow of $\varphi$. 
Since the cover $\varphi$ is branched over at least one point, we may write $\varphi_*R_{\varphi} = \sum_{i=1}^r a_i B_i$ for some integers $r,a_1,\ldots,a_r \geq 1$. By Proposition~\ref{prop:rcsceresa} and Lemma~\ref{lem:relcanpush}, to prove the theorem, it suffices to show that for a fixed $C$, and very general choice of branch points $(B_1,\ldots,B_r) \in C^r(\C)$, the point 
\[ \varphi_*(D_{\varphi}) = d (2g_{C}-2) \left( \sum_{i=1}^r a_i B_i \right) - \left(\sum_{i=1}^r a_i \right)dK_C \] 
in $J \colonequals \Jac(C)$ is nontorsion. We now produce a morphism $i_r \colon C^r \rightarrow J$ such that $\varphi_*(D_{\varphi}) = i_r(B_1,\dots,B_r)$; then if $X \colonequals i_r(C^r)$, proving that $\varphi_*(D_{\varphi})$ is nontorsion for a very general choice of $(B_1,\dots,B_r) \in C^r(\C)$ reduces to showing that a very general point of $X(\C)$ is nontorsion. 
To that end, let $i \colon C \rightarrow J$ be the Abel-Jacobi map defined by $x \mapsto d(2g_{C}-2)x-dK_{C}$, and let  $i_r \colon C^r \rightarrow J$ be  defined by $i_r(x_1,\ldots,x_r) \colonequals \sum_{i=1}^r a_i i(x_i)$.  

We now show that a very general point of $X(\C)$ is nontorsion. This is true if $X=J$, since $J_{\tors}$ is a countable union of closed subvarieties of $J$ of positive codimension. So now assume $X \neq J$. Since $\deg(dK_C) = d(2g_C-2) > 0$ when $g_C \geq 2$, a Riemann--Roch computation shows that the smallest abelian subvariety of $J$ containing $i(C)$ is $J$. Hence by Lemma~\ref{lem:curvesinJacobians}, every torsion translate of an abelian subvariety contained in $X$ has positive codimension in $X$. Furthermore, since Theorem~\ref{thm:manin-mumford} implies that the torsion locus in $X$ is contained in a finite union of torsion translates of abelian subvarieties of $J$ which are contained in $X$, it follows that a general point of $X(\C)$ is nontorsion when $X \neq J$.
\end{proof}

\begin{rem}
Note that the proof actually gives us the stronger conclusion that $\Cer(\widetilde{C})$ is nontorsion in the Chow group for a {\textit{general}} point $[\widetilde{C} \rightarrow C]$ in $ \mathcal{H}_{d,r}(C,\rho)(\C)$ when $\dim i_r(C^r) < \dim J$. This holds precisely when $r < \dim J$ since $a_i$ are positive (arguing for example, using the induced map on tangent spaces at a point as in the case when all $a_i=1$).
\end{rem}

\section{Infinitely many families with small Ceresa-torsion locus} 
In this section we prove \Cref{T:CerMMgen}, using \Cref{thm:RelManMum} applied to a pair $(\mathcal{A},X)$, where $\mathcal{A}$ is an abelian scheme $\mathcal{A} \rightarrow Z_r$ obtained by pulling back the universal family of relative Jacobians to a subvariety $Z_r \subset \mathcal{M}_{g,r}$, and where $X \subset \mathcal{A}$ is a subvariety obtained as the image of a certain weighted Abel--Jacobi map $i_r$ that we now define. To that end, recall the setup of \Cref{T:CerMMgen}.

Fix an integer $r \geq 1$.  For a locally closed subvariety $Z$ of $\mathcal{M}_g$, let $\mathcal{J}_g \rightarrow Z$ be the universal family of Jacobians of the curves parameterized by $Z$,  let $Z_r \colonequals \mathcal{M}_{g,r} \times_{\mathcal{M}_g} Z$, we let $\mathcal{A}_Z \colonequals Z_r {\times_Z} \mathcal{J}_g$.  Let $\pi_{Z} \colon \mathcal{A}_Z \rightarrow Z_r$ and $f_Z \colon Z_r \rightarrow Z$ denote the natural projection maps, and when $Z$ is fixed we drop the subscripts on $\mathcal{A}_Z, \pi_Z$ and $f_Z$ to ease notation. For an $r$-tuple of positive integers $\vec{a} \colonequals (a_1,a_2,\ldots,a_r)$, define the section $i_r^{\vec{a}}$ of  $\pi \colon \mathcal{A} \rightarrow Z_r$ by
\begin{align*}
    i_r^{\vec{a}} \colon Z_r  &\rightarrow \mathcal{A} \\
    (C,P_1,P_2,\ldots,P_r) &\mapsto \sum_{i=1}^r \left[ (2g-2) a_i P_i - a_i K_C \right].
\end{align*}

For an abelian scheme $\mathcal{A} \rightarrow U$, and a point $z \in U$, we let $\mathcal{A}_z$ denote the fiber of $\mathcal{A}$ over $z$. 
\begin{lem}\label{lem:secthroughpt}
Let $\tilde{z} \in Z_r(\Qbar)$ and let $z \colonequals f(\tilde{z}) \in Z(\Qbar)$. Let $\vec{a}$ be an $r$-tuple of positive integers, and let $i_r \colonequals i_r^{\vec{a}}$ as above. Then there exists an \'{e}tale open neighbourhood $U$ of $z$ in $Z$, and a section $s \colon U \rightarrow ({Z_r})_U$ of $f_U$ such that $s(z) = \tilde{z}$. Furthermore, if $\overline{\Z i_r(\tilde{z})} = \mathcal{A}_z$, then $\overline{\Z i_r(s(U))} = \mathcal{A}_U$.
\end{lem}
\begin{proof}
Since $\mathcal{M}_{g,r}$ and $\mathcal{M}_g$ are smooth Deligne--Mumford stacks over $\Spec \Z$ by \cite[Theorem~5.2]{DMstack}, we can pass to \'{e}tale neighbourhoods $V$ of $\tilde{z}$ and $U$ of $z$ respectively and assume that $f \colon V \rightarrow U$ is a smooth surjective morphism of schemes. The existence of $s$ now follows from \cite[Corollaire~17.16.3~(ii)]{EGAIV4} and the proof of \cite[Corollaire~17.16.1]{EGAIV4}.

Let $\mathcal{B} \colonequals \overline{\Z i_r(s(U))}$, and let $B \colonequals \overline{\Z i_r(\tilde{z})}$. Since $z \in U$ and $s(z) = \tilde{z}$, it follows that $i_r(\tilde{z}) \in i_r(s(U))$ and hence $B \subset \mathcal{B}_z \subset \mathcal{A}_z$. Assume $B = \mathcal{A}_z$. Then $B = \mathcal{B}_z =  \mathcal{A}_z$.  We want to show that $\mathcal{B} = \mathcal{A}_U$. If $\dim \mathcal{B} < \dim \mathcal{A}_U = \dim \mathcal{A}$, then $\dim \mathcal{B}_z < \dim \mathcal{A}_z$, contradicting $\mathcal{B}_z = \mathcal{A}_z$. So $\dim \mathcal{B} = \dim \mathcal{A}$, and in this case $\mathcal{B}$ has finite index in $\mathcal{A}_U$. Since $\mathcal{B}$ is a closed subgroup variety of $\mathcal{A}_U$, if $\mathcal{B}$ has finite index in $\mathcal{A}_U$, it is also open in $\mathcal{A}_U$ and since $\mathcal{A}_U$ is irreducible, it would also follow that $\mathcal{B} = \mathcal{A}_U$. 
\end{proof}

\begin{prop}\label{prop:AJgenerates} 
Assume $g \geq 2, r \geq 1$ and that $C$ is defined over $\Qbar$. Let $Z,\vec{a}, Z_r, \mathcal{A}, \pi, i_r$ be as in \Cref{lem:secthroughpt}. Then as we vary over nonempty \'{e}tale open subsets $U$ of $Z$, there are infinitely many sections $s \colon U \rightarrow (Z_r)_U$ of $f_U$ such that
$\overline{\Z i_r(s(U))} = \mathcal{A}_U$.
\end{prop}
\begin{proof}
Let $C$ be a curve corresponding to a point $z \in Z$. Since $C$ is defined over $\Qbar$, by \Cref{lem:infintepoints}, there are infinitely many points $\tilde{z} \in (Z_r)_z(\Qbar)$ corresponding to an $r$-tuple of pairwise distinct points $(P_1,P_2,\ldots,P_r) \in C^r$ such that $\overline{\Z i_r(\tilde{z})} = \Jac(C)$. Choosing $s$ as in \Cref{lem:secthroughpt} for each such $\tilde{z}$ now finishes the proof.
\end{proof}

\begin{proof}[Proof of \Cref{T:CerMMgen}]
Let $\pi \colon \mathcal{A} \rightarrow Z_r$ be the abelian scheme  as in \Cref{prop:AJgenerates}. 
For each $i$ with $1 \leq i \leq r$, let $P_i$ be the $i$th branch point for a degree $d$ cover  $\varphi \colon \widetilde{C} \rightarrow C$ with monodromy representation $\rho$ and ramification divisor $R_{\varphi}$. Define $a_i$ to be the coefficient of $P_i$ in the divisor $d(2g-2)\varphi_*(R_{\varphi}) \in \Pic^0(C)$, and let $i_r \colonequals i_r^{\vec{a}}$ be as in \Cref{prop:AJgenerates}. Then by the definition of $i_r$ and \Cref{lem:relcanpush}, it follows that if $D_{\varphi}$ is the relative canonical shadow for a cover $\varphi$ as above, then $\varphi_*(D_{\varphi}) = i_r(P_1,\ldots,P_r)$.

By \Cref{prop:AJgenerates}, as we vary over nonempty \'{e}tale open subsets $U$ of $Z$, there are infinitely many sections $s \colon U \rightarrow (Z_r)_U$ such that  that $\overline{i_r(s(U))} = \mathcal{A}_U$. Fix such a section $s$. Since $\dim i_r(s(U)) = \dim U = \dim Z< g$ by assumption and since $\overline{\Z i_r(s(U))} = \mathcal{A}_U$ by choice of $s$, \Cref{thm:RelManMum} implies that $i_r(s(U)) \cap \mathcal{A}_{\tors}$ is contained in a Zariski closed subscheme of $X$ of positive codimension, or equivalently that for $x_0 \colonequals (C,P_1,\ldots,P_r)$ in a nonempty \'{e}tale open subset $U$ of $Z$, the point $i_r(P_1,\ldots,P_r)$ is nontorsion in the Jacobian of $C$. By \Cref{prop:rcsceresa} and the equality $\varphi_*(D_{\varphi}) = i_r(P_1,\ldots,P_r)$, it follows that $\Cer(\widetilde{C})$ is nontorsion for any $[\widetilde{C} \rightarrow C]$ in $\mathcal{H} \cap \tar^{-1}(U) \subset \tar^{-1}(Z)$. Since $\tar$ is a finite flat morphism, it follows that  $\mathcal{H} \cap \tar^{-1}(U)$ is a nonempty open dense  subset of the irreducible subvariety $\mathcal{H}$. 
\end{proof}

\section{Explicit families of curves with small Ceresa torsion locus}\label{sec:one-parameter}

\subsection{A $1$-parameter $D_{12}$-family.}

Recall the $1$-parameter family of smooth projective  genus $6$ curves in \Cref{T:CerMM2}
with affine model given by
\[ C_t: y^{12} = \left(\frac{x + 1}{x - 1} \right)^3 \left(\frac{x + t}{x - t} \right)^4.\] 
Fix a primitive $24$th root of unity $\zeta_{24}$, and for $m \mid 24$, let $\zeta_m \colonequals \zeta_{24}^{24/m}$. There is an action of the dihedral group $D_{12}$ on $C_t$ generated by the automorphisms 
\[ \sigma(x,y) \colonequals (- x, y^{- 1}), \quad \quad \tau(x, y) \colonequals (x, \zeta_{12}y).\] 

We prove Theorem~\ref{T:CerMM2} in this subsection, namely that there are at most finitely many $t \in \C$ for which $\Cer(C_t)$ is torsion. The strategy is to push forward the relative canonical shadows for the covers $C_t \rightarrow C_t/\langle a \rangle$, with $a \in \{\sigma,\tau^3 \sigma\}$ to the quotient $C_t/\langle \tau^6,\tau^3 \sigma \rangle$ to produce two linearly independent sections of an elliptic fibration, and  then appeal to the Masser--Zannier unlikely intersections theorem for finiteness of the simultaneous torsion locus of these sections.

\begin{center}
\begin{tikzcd}[%
        cells={shape=rectangle}
        ]%[row sep=8pt]
   &[-12pt] C_t \arrow[out=185,in=165
   ,loop,looseness=10, "\tau"]
   \arrow[dl, "\varphi" above, end anchor=100] \arrow[dr, "\varphi'" , end anchor=110] \arrow[d, "\psi"]
   &[-12pt]&[12pt]  &[-12pt] \langle e \rangle  \arrow[dl, no head, end anchor=90] \arrow[dr, no head, end anchor=90] \arrow[d, no head] &[-12pt] \\
   \begin{matrix}  X_t \colonequals \\ C_t/ \langle\sigma\rangle \end{matrix} & \begin{matrix}  Y_t
   \colonequals \\ C_t/\langle\tau^6\rangle \end{matrix} 
   \arrow[out=170,in=160,loop,looseness=10, "\overline{\tau}" {yshift=-5pt}, shift right=3]\arrow[d, "\pi_E"]
   & \begin{matrix} X_t' \colonequals \\  C_t/ \langle{\tau\sigma}\rangle\end{matrix}  & \begin{matrix}  \langle \sigma \rangle\\ \simeq \Z/2\Z \end{matrix}
   %\arrow[rdd, no head]
   & \begin{matrix}\langle \tau^6 \rangle\\ \simeq \Z/2\Z \end{matrix} 
   %\arrow[d, no head]  
   \arrow[d, no head]
   &\begin{matrix} \langle \tau\sigma \rangle\\ \simeq \Z/2\Z \end{matrix} 
   %\arrow[ldd, no head]
   \\
   &\begin{matrix} E_t:=\\C_t/\langle \tau^6, \tau^3 \sigma \rangle \end{matrix} 
   &&&\begin{matrix} \langle \tau^6, \tau^3 \sigma \rangle \\ \simeq \Z/2\Z \times \Z/2\Z\end{matrix} 
\end{tikzcd}   
\end{center}

We begin by writing down explicit affine models for the relevant quotient curves. Let $\varphi, \varphi'$ and $\psi$ be maps from the curve $C_t$ to its quotients $X_t,X_t',Y_t$ by $\sigma, \tau \sigma$ and $\tau^6$ respectively.  Since $\tau^6$ generates the center of $D_{12}$, the action of $\sigma$ and $\tau$ descend to automorphisms $\bar{\sigma}$ and $\bar{\tau}$ of the quotient $Y_t$. Consider the automorphism $\overline{\tau}^3 \overline{\sigma}$ of $Y_t$ and let $\pi_{E} \colon Y_t \rightarrow Y_t/\langle \bar{\tau}^3\bar{\sigma} \rangle \equalscolon E_t$ denote the corresponding quotient map. We now write down explicit equations for $Y_t$ and $E_t$.

Define rational functions $y_1,u,w$ by
\[ y_1 \colonequals y^2, \quad u \colonequals y_1^2 \left(\frac{x - 1}{x + 1}\right)\left(\frac{x - t}{x + t}\right), \quad \textup{ and }   w \colonequals \left(\frac{x+t}{x-t}\right).\]
Then $y_1,u,w$ are fixed by $\tau^6$ and $Y_t \colonequals C_t / \langle \tau^6 \rangle$ is defined by the affine equation
\[ y_1^6 = \left(\frac{x + 1}{x - 1}\right)^3
\left(\frac{x + t}{x - t}\right)^4, \text{ or equivalently by  } u^3=w. \]

Similarly, define rational functions $v,z$ by
\[ v \colonequals u+u^{-1}, \quad z \colonequals \left(\frac{y_1-y_1^{-1}}{R(v)}\right)h(v), \textup {\quad where} \] 
\begin{align*} 
%S(v) &\colonequals (v^3t + v^3 - v^2t - v^2 - 2vt - 2v + 2t) \\ 
R(v) &\colonequals v(v-1)(v+2)(t+1)-2 \\ h(v) &\colonequals (v^3-3v+2)t^2-(v^3-3v-2).\end{align*} 
Since $\tau^3\sigma(y_1) = -y_1^{-1}$ and $\tau^3 \sigma(u) = u^{-1}$, it follows that  $v,z$ are invariant under both $\tau^6$ and $\tau^3 \sigma$, and an affine equation for the genus $1$ curve $E_t \colonequals Y_t / \langle \bar{\tau}^3 \bar{\sigma} \rangle$ is given by
\[ z^2 = (v-2) h(v). \]
We make $E_t$ an elliptic curve over $\Q(t)$ by choosing the point $(v=2,z=0)$ as the identity. 

Fix an element $c$ in an algebraic closure of $\overline{\mathbb{Q}(t)}$ that satisfies $ c^3 = \frac{1+t}{1-t}$, and for $i=1,2,3$, consider the points of $Y_t$ defined by
\begin{align*} A_{i,t} &\colonequals \left(x=1,(x-1)y_1^2=2\zeta_3^i c^4 \right), \ \textup{or equivalently by}\ \left( u = \zeta_3^i c, w = \frac{1+t}{1-t}  \right), \textup{ and, } \\
B_{i,t} &\colonequals \left(x=-1,\frac{y_1^2}{(x+1)}=\frac{-\zeta_3^i}{2c^4} \right), \  \textup{or equivalently by}\ \left( u = \frac{\zeta_3^i}{c}, w = \frac{1-t}{1+t} \right) .\end{align*}

\begin{lem}\label{lem:branchppoints}
The ramification divisors $R_{\psi}, R_{\varphi}, R_{\varphi'}, R_{\pi_{E}}$ for   $\psi, \varphi, \varphi'$ and $\pi_E$ are
\begin{align*} 
R_{\psi} &= \sum_{i=1}^3 \left( \left(x=1,(x-1)y^4=2\zeta_3^i c^4 \right) + \left(x=-1,\frac{y^4}{(x+1)}=\frac{-\zeta_3^i}{2c^4} \right) \right)
\\
R_{\varphi} &= (1/x=0, y=1) + (1/x=0,y=-1), \\
R_{\varphi'} &= (x=0, y=\zeta_{24}) + (x=0,y=-\zeta_{24}), \textup{ and,} \\
R_{\pi_{E}} &= (x = 0, y_1 = \zeta_{4}) + (x = 0, y_1 = - \zeta_{4}).
\end{align*}

\end{lem}
\begin{proof}
These follow from the explicit equations for the curves $C_t,Y_t$ above, and the fact that the ramification points of $\varphi,\varphi', \pi_E$ are the fixed points of the automorphisms $\sigma(x,y) =(-x,y^{-1}), \tau \sigma(x,y) = (-x,\zeta_{12}y^{-1})$ and $\overline{\tau}^3 \overline{\sigma}(x,y_1) = (-x,-y_1^{-1})$ respectively. 
\end{proof}

\begin{lem}\label{lem:Ctrelcanonicalshadow}
Let $D_{\varphi}$ and $D_{\varphi'}$ be the relative canonical shadows for $\varphi$ and $\varphi'$.
Then
\begin{align*}\psi_*(D_\varphi) &= 20 \left(\frac{1}{x}=0,y_1=1 \right)-4K_{Y_t}-2\sum_{i=1}^3 (A_{i,t}+B_{i,t}), \textup{\  and,}\\
\psi_{*}(D_{\varphi'}) &= 20(x=0, y_1=\zeta_{12})-4K_{Y_t}-2\sum_{i=1}^3 (A_{i,t}+B_{i,t}).\end{align*}

\end{lem}
\begin{proof}
Since $g_{C_t}=6$ and since $\varphi$ is Galois, using the formula in Remark~\ref{rem:RCspecial}~\eqref{rem:Gal}, we get
\[D_{\varphi}=10 R_{\varphi}-2K_{C_t}.\]

Since $y_1 = y^2$ and since $R_{\varphi} = (\frac{1}{x}=0, y=1) + (\frac{1}{x}=0, y=-1)$ by \Cref{lem:branchppoints}, it follows that \[\psi_{*}(R_{\varphi}) = 10 \psi_{*} R_{\varphi}-2 \psi_{*}(K_{C_t}) = 20 \left(\frac{1}{x}=0, y_1=1 \right) -2 \psi_{*}(K_{C_t}).\] Since $\deg \psi = 2$, the Riemann--Hurwitz formula and the projection formula imply that  
\[ \psi_*(K_{C_t}) = \psi_*(\psi^*(K_{Y_t}) + R_{\psi}) = 2K_{Y_t} + \psi_*(R_{\psi}).\]  Putting everything together we get
$$\psi_*(D_{\varphi}) = 20 \left(\frac{1}{x}=0, y_1=1 \right) - 4K_{Y_t} - 2\psi_*(R_{\psi}).$$
The desired expression for $\psi_*(D_{\varphi})$ now follows from $y_1=y^2$ and \Cref{lem:branchppoints}. The expression for $\psi_{*}(D_{\varphi'})$  can similarly be obtained using the formula for $R_{\varphi'}$ (\Cref{lem:branchppoints}) in place of $R_{\varphi}$.
\end{proof}

\begin{lem}\label{lem:Ctshadowspushforward}
Let $D_1(t) \colonequals (\pi_{E})_* (\overline{\tau}^2)_* \psi_*(D_{\varphi})$ and $D_2(t) \colonequals (\pi_{E})_*\psi_*
(D_{\varphi'})$. Then $D_1(t)$ and $D_2(t)$ are linearly independent sections of the elliptic fibration $E_t$.
\end{lem}

\begin{proof}
We compute $D_1(t)$ and $D_2(t)$ as a sum of points on $E$ expressed using the $(v,z)$-coordinates on $E$. If $x=0$ and $y=\pm \zeta_{4}$, then $y_1^2 = u = -1, w = -1, v = -2, h(v) = 4, z = \pm  4\zeta_4$.  
Combining this with the Riemann--Hurwitz formula and \Cref{lem:branchppoints}, we get 
\[(\pi_{E})_*(K_{Y_t}) = (\pi_{E})_*(R_{\pi_{E}})=(-2,4\zeta_4)+(-2,-4\zeta_4).\] 
If $x = \pm 1$, then $w = \frac{\pm 1+t}{\pm 1-t}, u = (\zeta_3^i c)^{\pm 1}, v = \zeta_3^i c + \frac{1}{c}\zeta_3^{-i}, h(v) = 0$  and $z=0$.
Combining this with $\pi_E(x=0,y_1=\zeta_{12}) = (1,-2 \zeta_4)$, \Cref{lem:Ctrelcanonicalshadow} and the definition of $D_2(t)$,  we get  
\[ D_2(t) = 20(1,-2\zeta_4) - 4(-2,4\zeta_4) - 4(-2,-4\zeta_4)- 4\sum_{i=0}^2 \left(c \zeta_3^i +\frac{1}{c\zeta_3^{i}},0 \right)
\] 
Combining this with

\begin{equation}\label{eqn:linearequiv} \begin{split} \divi \left(\frac{(v+2)^4}{(v-2)^2} \right) = 4(-2,4\zeta_4) + 4(-2,-4\zeta_4)-8(2,0), \text{ and,}\\
\divi \left(\frac{z^2}{(v-2)^3} \right) = 2\sum_{i=0}^2 \left(c \zeta_3^i +\frac{1}{c\zeta_3^{i}},0 \right) - 6(2,0), \end{split}\end{equation}
we get that 
\[ D_2(t) =20(1,-2\zeta_4) - 20(2,0) \in \Pic^0(E_t).\] 

We now compute $D_1(t).$ As $\overline\tau^2(x,y_1)
=(x,\zeta_3 y_1)$, it follows that $\overline{\tau}^2$ cyclically permutes the elements of the two sets $\{A_{1,t},A_{2,t},A_{3,t} \}$ and $\{B_{1,t},B_{2,t},B_{3,t} \}$. Since $\overline{\tau^2}$ is an automorphism and $K_{Y_t}$ is the canonical divisor, we further have $(\overline{\tau}^2)_*(K_{Y_t}) = K_{Y_t}$ in $\Pic^0(Y_t)$. Combining this with the expression for
$\psi_{*}(D_{\varphi})$ from Lemma \ref{lem:Ctrelcanonicalshadow}, we get
\[(\bar{\tau}^2)_*(\psi_*(D_\varphi)) = 20 \left(\frac{1}{x} = 0, y_1=\zeta_3 \right) - 4K_{Y_t}-2\sum_{i=1}^3A_{i,t}-2\sum_{i=1}^3B_{i,t} \in \Pic^0(Y_t).\]

Combining this with $\pi_{E}(1/x=0,y_1=\zeta_3) = (v=-1,z = 2t(2 \zeta_3+1))$, we get
\begin{align*}(\pi_{E})_*(\bar{\tau}^2)_*(\psi_*(D_\varphi)) &= 20(-1, 2t(2\zeta_3+1)) - 4(-2, 4\zeta_4)-4(-2,-4\zeta_4)-4\sum_{i=0}^2 (c \zeta_3^i +\frac{1}{c\zeta_3^{i}},0). \\ &= 20(-1, 2t(2\zeta_3+1)) - 20(2,0) \in \Pic^0(E_t), \text{ using \Cref{eqn:linearequiv}.}\end{align*}

We use code \cite{code} in Magma \cite{magma} to check that $D_1(t)$ and $D_2(t)$ are linearly independent by checking that the specializations to $t=7/9$ are linearly independent on $E_{7/9}$. 
\end{proof}

\begin{rem}
Since $\pi_{E}(1/x=0,y_1=1) = (v=2,z=0)$, a similar computation can be used to show that  $(\pi_{E})_* (\psi_*)(D_{\varphi}) = 0$, so the additional pushforward $(\overline{\tau}^2)_*$ in the definition of $D_1(t)$ is necessary to produce linearly independent sections of $E_t$.
\end{rem}

\begin{lem}\label{lem:StollTors} Let $D_1(t),D_2(t),E_t$ be as in \Cref{lem:Ctshadowspushforward}. There are no parameters $t \in \C \setminus \{0,\pm 1 \}$ such that $D_1(t)$ and $D_2(t)$ are both torsion points of $E_t$.
\end{lem}
\begin{proof} Let $P_1(t) \colonequals (-1, 2t(2\zeta_3+1)) - (2,0)$, and $P_2(t) \colonequals (1,-2\zeta_4) - (2,,0)$. The proof of \Cref{lem:Ctshadowspushforward} shows that $D_1(t) = 20P_1(t)$, and $D_2(t) = 20P_2(t)$. We use \cite[Corollary~22]{Stoll} to check that the simultaneous torsion locus for the sections $D_1(t)$ and $D_2(t)$ is empty. For the Legendre family $E_{\lambda} \colon y^2=x(x-1)(x-\lambda)$, and for a pair $\alpha,\beta \in \C \setminus \{0,1\}$ with $\alpha \neq \beta$ and such that $\trdeg \Q(\alpha,\beta) = 1$, Stoll collects a list   of absolutely irreducible polynomials $F$ in $\Z[a,b]$ (see \cite{StollCurves}). If $\alpha$ and $\beta$ are $x$-coordinates of points that are simultaneously torsion on $E_{\lambda}$ for some $\lambda$, then $F(\alpha,\beta) = 0$ for some polynomial $F$ in his list. 

To apply his results, we transform the given equation $E_t \colon z^2=(v-2)h(v)$ to Legendre form and compute the $x$-coordinates $\alpha,\beta$ for the sections $P_2(t)$ and $P_1(t)$. It is then straightforward to check that $\trdeg \Q(\alpha,\beta) = 1$ and compute an absolutely irreducible $F \in \Z[a,b]$ such that $F(\alpha,\beta) = 0$. Since Stoll proves that such a polynomial $F$ divides one of the polynomials in his list of absolutely irreducible polynomials, it suffices to check that the $F$ that we compute is not in Stoll's list. This was indeed the case for the relevant triple $(\alpha,\beta, F)$ listed below coming from the sections $D_1(t),D_2(t)$ and this proves the lemma. 

%To describe $(\alpha,\beta,F)$, first define $\psi$ to be the fractional linear transformation defined by $2 \mapsto \infty, v_1 \mapsto 0, v_2 \mapsto 1$, where for $1 \leq i \leq 3$, we have $v_i \colonequals \zeta_3^i c + \frac{1}{c} \zeta_3^{-i}$ for an element $c$ satisfying that $c^3=(1+t)/(1-t)$. Then,

To find $(\alpha, \beta, F)$, first let $c \in \overline{\Q(t)}$ as before be an element such that $c^3=(1+t)/(1-t)$, and for $1 \leq i \leq 3$, let $v_i \colonequals \zeta_3^i c + \frac{1}{c} \zeta_3^{-i}$. If $\psi$ is the fractional linear transformation defined by $2 \mapsto \infty, v_1 \mapsto 0, v_2 \mapsto 1$, then
\[ \alpha = \psi(1)=\frac{(c+\zeta_3)(c-\zeta_3)^2}{(\zeta_3^2-\zeta_3)c(c-1)}, \quad \quad
\beta = \psi(-1) = \frac{(2\zeta_3+1)(c-\zeta_3)^3}{9c(c+1)}, \textup{ and,} \]
\[ F(a,b) \colonequals 2ab(a - 3b)^2 - (a - 3b)(a^2 - 6ab - 3b^2) + 8b^2 \in \Z[a,b].\qedhere\] 
\end{proof}

\begin{proof}[Proof of Theorem \ref{T:CerMM2}]
By \Cref{prop:rcsceresa}, it suffices to prove that $D_1(t)\times D_2(t)\in E_t\times E_t$ is nontorsion for every $t\in\C \setminus \{0,\pm 1\}.$ To that end, let $\mathcal{E}\rightarrow U\subset \mathbb{A}^1$ be the elliptic fibration with generic fiber $E_t$ and identity section $(v=2,z=0)$. Let $X$ be the image of the section $U\rightarrow \mathcal{E}\times\mathcal{E}$ such that $t\mapsto D_1(t)\times D_2(t)$. By Lemma \ref{lem:Ctshadowspushforward}, we see that $\overline{\Z X}=\mathcal{E}\times \mathcal{E}$. As $\dim X = 1$, we conclude by Theorem \ref{thm:RelManMum} and  Remark \ref{rem:MasserZannier} that the sections $D_1(t)$ and $D_2(t)$ are simultaneously torsion for at most finitely many values of $t\in \C \setminus \{0, \pm 1\}$. \Cref{lem:StollTors} then shows that this simultaneous torsion locus is infact empty.
\end{proof}

\subsection{Covers of a $2$-parameter \text{$S_3$} family.}
Recall the setup of Theorem~\ref{T:CerMM1} -- we have a $2$-parameter family of plane quartic genus $3$ curves cut out by 
\[ f_{u,w}(X,Y,Z) \colonequals w^2(X(Y^3+Z^3)+Y^2Z^2)+(u^3+w^4)X^2YZ+w^2u^3X^4=0, \]
and sections $A \colonequals [1:w:-w], A' \colonequals [1:-u:0]$ of the projection $p \colon V(f) \rightarrow \mathbb{A}^2_{(u,w)}$. Recall we have the subset $U$ of $\mathbb{A}^2_{u,w}$ that was the image of the smooth locus of $p$, a moduli map $U \rightarrow \mathcal{M}_{3,2}$ corresponding to $(p,A,A')$, the subvariety $S \subset \mathcal{M}_{3,2}$ the image of the moduli map, the  subvariety $\mathcal{H}$ that is an irreducible component of $\tar^{-1}(S)$. In this subsection, we prove that $\Cer(\widetilde{C})$ is nontorsion for $[\widetilde{C} \rightarrow C]$ in an open dense subset of $\mathcal{H}(\C)$. 

For this, first recall that the $S_3$ action on $\mathcal{C}$ gives rise to an isogeny decomposition of the Jacobian as $\mathcal{E} \times \mathcal{E} \times \mathcal{E}'$, where $\mathcal{E}$ factors are quotients by two of the involutions, and $\mathcal{E}'$ is the quotient by an order $3$ element in $S_3$ (see for example \cite[Table~2]{DecompJac}). The strategy is to show that the relative canonical shadow for the double cover generates an abelian scheme of dimension at least $3$ by showing that its projection to $\mathcal{E}^2$ and $\mathcal{E}'$ are both dominant, and that $\mathcal{E}$ and $\mathcal{E}'$ are nonisogenous. Since the Ceresa torsion locus is contained in the torsion locus for the relative canonical shadow, \Cref{T:CerMM1} follows from the relative Manin--Mumford conjecture.

For ease of notation, we write $t=(u,w)$, and let $\mathcal{C}_t, A_t,A'_t$ be the corresponding curve and marked points. The relative canonical divisor $K_\mathcal{C}$ for the fibration $\mathcal{C} \rightarrow U$ is given by any hyperplane section of $\mathbb{P}^2_U$. Let
$K_{\mathcal{C}}=\divi(X)=2[0:1:0]+2[0:0:1]$. Let $\pi_{\mathcal{E}} \colon\mathcal{E} \rightarrow \mathbb{A}^2_{\Q}$ and $\pi_{\mathcal{E}'} \colon \mathcal{E}' \rightarrow \mathbb{A}^2_{\Q}$ be genus $1$ fibrations in $\proj \Q[a,b,c] \times \Spec \Q[u,w]$ cut out by the equations
\[w^2(a^3-3abc+b^2c)+(u^3+w^4)bc^2+w^2u^3c^3 = 0, \text{ and,}\]
\[ w^2(a^2c+b^3+ab^2)+(u^3+w^4)abc+w^2u^3ac^2 = 0 \]  
respectively. These genus $1$ fibrations become elliptic with the choice of identity sections $[a:b:c]=[0:1:0]$ and $[a:b:c]=[1:0:0]$ for $\pi_{\mathcal{E}}$ and $\pi_{\mathcal{E}'}$ respectively.

Let $\zeta_3$ be a primitive third root of unity. The curves in the family $\mathcal{C}$ admit three non-commuting involutions $\sigma_i$ and quotient maps $\psi_i \colon \mathcal{C} \to \mathcal{E} \colonequals \mathcal{C}/\sigma_i$ for $i=0,1,2$ given by 
\begin{equation}\label{E:quot}
\sigma_i([X:Y:Z]) \colonequals [X:\zeta_3^i Z:\zeta_3^{2i} Y], \hspace{2mm }\psi_i([X:Y:Z]) \colonequals [X(\zeta_3^i Y+\zeta_3^{2i} Z):YZ:X^2]
\end{equation}
Under the above maps, the points with $X=0$ map to the identity $[0:1:0]$ in $\mathcal{E}$.
Moreover, $\mathcal{C}$ admits  an order $3$ automorphism $\tau$ with corresponding quotient map $\mathcal{C}\rightarrow \mathcal{E}^{\prime} \colonequals \mathcal{C}/\tau$:
\begin{equation}\label{E':quot}
\tau([X:Y:Z]):=([X:\zeta_3 Y:\zeta_3^2 Z]), \hspace{2mm} \psi_{\tau}([X:Y:Z]):=[Y^3:XYZ:X^3]   
\end{equation}

%For $t=(u,w)$, we let $A_{t} \colonequals [1:w:-w], A'_{t} \colonequals [1:-u:0]$ be two disjoint sections of the fibration $\pi_{\mathcal{C}}$. 
The main result of this section is the following  reformulation of Theorem~\ref{T:CerMM1}.

\vspace{2mm}

\begin{center}
     \begin{tikzcd}[row sep=18pt]
   &\mathcal{\widetilde{C}}_t \arrow[d,"\overset{\textup{branched above }}{A_t \text{ and }A'_t}","\pi"'] &\\
	&\mathcal{C}_t\arrow[dr,"\psi_\tau"]\arrow[d,"\psi_2"]\arrow[dl,"\psi_1"']&\\
\mathcal{E}_t:= \mathcal{C}_t/\langle{\sigma_2}\rangle &\mathcal{E}_t:= \mathcal{C}_t/\langle{\sigma_1}\rangle& \mathcal{E}_t':=\mathcal{C}_t/\langle{\tau}\rangle\\
     \end{tikzcd}
     \end{center}

Let $ \mathcal{\widetilde{\mathcal{C}}}_{\univ} \rightarrow \mathcal{H}_{2,2}(3)$ denote the universal family obtained by pulling back the universal family above $\mathcal{M}_6$ along the source map $\mathcal{H}_{2,2}(3) \rightarrow \mathcal{M}_6$. Recall that we have a moduli map $U \rightarrow S \subset \mathcal{M}_{3,2}$ corresponding to the family $\mathcal{C}$. Let $i \colon \mathcal{H} \rightarrow \mathcal{H}_{2,2}(3)$ be the inclusion of the irreducible component $\mathcal{H}$ of $\tar^{-1}(S)$. We will abuse notation and continue to denote the pullback of the universal family $\widetilde{\mathcal{C}}_{\univ} \rightarrow \mathcal{H}_{2,2}(3)$ by $i$ as $\widetilde{\mathcal{C}}_{\univ} \rightarrow \mathcal{H}$.

\renewcommand{\thmcermmpream}{Let $\mathcal{C}, A, A'$ be as above.  Let $ \widetilde{\mathcal{C}} \rightarrow U$ be the family of genus $6$ curves obtained by pulling back the composition $\widetilde{\mathcal{C}}_{\univ} \rightarrow \mathcal{H} \rightarrow S$ along the moduli map $U \rightarrow S$. }

\TCerMMone*

We first prove some lemmas. Let $\Pic^0(\mathcal{C}/U) \equalscolon \mathcal{J} \rightarrow U$ denote the  family of Jacobians of the family of curves $\mathcal{C} \rightarrow U$. We define $\pi_{*}(D_{\pi}) \colon U\rightarrow \mathcal{J}$ to be the section whose value at $t \in U$ is given by the pushforward of relative canonical shadow of $\pi_t \colon \widetilde{\mathcal{C}}_t\rightarrow \mathcal{C}_t$ to $\mathcal{J}_t$. For $i=0,1,2,$ let $P_i \colon U\rightarrow \mathcal{E}$ be the sections given by $(\psi_i)_*(\pi)_{*}(D_{\pi})$. Denote by $P_\tau \colon U\rightarrow \mathcal{E}^{\prime}$ the section given by $(\psi_\tau)_*(\pi_{*})(D_{\pi})$. The key to proving \Cref{T:CerMM1} is the following lemma.  
\begin{lem}\label{lem:sectionsindependent}
The elliptic fibrations $\pi_{\mathcal{E}}$ and $\pi_{\mathcal{E}^{\prime}}$ are geometrically non-isogenous. Furthermore, $P_1$ and $P_2$ are linearly independent sections of $\pi_{\mathcal{E}}$, while  $P_{\tau}$ is an  infinite order section of $\pi_{\mathcal{E}'}$.
\end{lem}

\begin{proof}
 Using Lemma \ref{lem:relcanpush} and the Riemann--Hurwitz formula for the maps $\psi_i$, we compute the expressions for $P_i$ for $i=1,2,\tau$, and  we get
\begin{equation}
 P_i= 8 (\psi_i)_*(A_{t}+A'_{t})-4(\psi_{i})_* (K_{\mathcal{C}_t}).
\end{equation}

It suffices to find a point $t \in U$ such that the specializations $\overline{P_1}$ and $\overline{P_2}$ of $P_1$ and $P_2$ at $t$ are linearly independent, the specialization $\overline{P_\tau}$ of $P_{\tau}$ at $t$ is nontorsion, and such that the specializations $E$ and $E'$ of the fibrations $\mathcal{E}$ and $\mathcal{E}'$ at $t$ are geometrically non-isogenous. We use code \cite{code} in Magma \cite{magma} to check these at $t=(w,u)=(1,2)$ using
\begin{align*}
\overline{P_1} &= 8[(\zeta_3-\zeta_3^2): -1:1]+8 [-2\zeta_3:0:-1]- 8\cdot[0:1:0]\\
\overline{P_2} &= 8[(\zeta_3^2-\zeta_3): -1:1]+8[-2\zeta_3^2:0:-1]- 8\cdot[0:1:0]\\
\overline{P_\tau} &= 8[1: -1:1]+8[-8:0:1]- 8\cdot[1:0:0]-8\cdot[1:-1:0], \qedhere 
\end{align*}
\end{proof}

\begin{cor}\label{cor:ZarClosDimension}
Consider the family of abelian schemes $\mathcal{E}^2\times \mathcal{E}'\rightarrow U$ of relative dimension $3$.  Let $X\subset \mathcal{E}^2\times \mathcal{E}'$ be the image of the section $P_1\times P_2\times P_{\tau} \colon U \rightarrow \mathcal{E}^2\times \mathcal{E}'$. Then $\overline{\Z X} = \mathcal{E}^2\times \mathcal{E}'$.
\end{cor}
\begin{proof} Let $\pi_{1} \colon \mathcal{E}^2\times \mathcal{E}' \rightarrow \mathcal{E}^2$ be the projection onto the first factor. The sections $P_1$ and $P_2$ are linearly independent by Lemma \ref{lem:sectionsindependent}. Remark \ref{rem:MasserZannier} then implies that $ \pi_1(\overline{\Z X}) = \overline{\Z (P_1 \times P_2)(U)} = \mathcal{E}^2$. Lemma \ref{lem:sectionsindependent} implies that $P_\tau$ is nontorsion and hence $\overline{\Z P_\tau(U)}=\mathcal{E}'$. Combining the last two sentences with the fact that $\pi_{\mathcal{E}}$ and $\pi_{\mathcal{E}'}$ are non-isogenous, we see that $\overline{\Z X}=\mathcal{E}^2\times \mathcal{E}'$.
\end{proof}

\begin{proof}[Proof of \Cref{T:CerMM1}] Since $\mathcal{H}$ is irreducible, any nonempty open subvariety of $\mathcal{H}$ is Zariski dense. Since $\tar \colon \mathcal{H} \rightarrow U$ is a finite flat morphism, the preimage of a nonempty open subset of $U$ is once again a nonempty open subset of $\mathcal{H}$. Combining the last two sentences with \Cref{prop:rcsceresa}, we see that it is enough to prove that the pushforward of the relative canonical shadow of $((\psi_1\times\psi_2\times\psi_\tau) \circ \pi)_*(D_{\pi}): U \rightarrow \mathcal{E}^2 \times \mathcal{E}'$ is nontorsion in $\mathcal{E}^2 \times \mathcal{E}'$ on an open dense subset of $U$. Let $X$ be as in Corollary~\ref{cor:ZarClosDimension}.
Then $\overline{\Z X} = \mathcal{E}^2\times \mathcal{E}'$ by Corollary~\ref{cor:ZarClosDimension}.   
As $\dim U=\dim X = 2<3 = \dim (\mathcal{E}^2\times \mathcal{E}')-\dim U$, by Theorem \ref{thm:RelManMum}, we conclude that $X(\C)\cap (\mathcal{E}^2\times \mathcal{E}')_{\tors}$ is a Zariski closed subset of $X$ of positive codimension. Since $X$ is irreducible, the complement of the torsion locus of $X(\C)$ is Zariski dense, or equivalently that the relative canonical shadow is nontorsion on an open dense subset of $U$. 
\end{proof}

\begin{rem}\label{rem:S3stratum}
In \cite[Proposition~5.3]{AriJef2}, Laga and Shnidman show that a very general point of the moduli space $\mathcal{M}_3^{S_3}$ of genus $3$ curves with an $S_3$ action has nontorsion Ceresa cycle. Since a curve dominating a curve with infinite order Ceresa cycle also has infinite order Ceresa cycle \cite[Proposition~25]{CerMod}, Laga and Shnidman's result implies that $\Cer(\widetilde{\mathcal{C}_t})$ is nontorsion for a very general $t \in U$. Although our proof technique does not directly tell us about the Ceresa cycles for the curves in the family $\mathcal{C}\rightarrow U$ (see Remark~\ref{rem:RCspecial}~\eqref{rem:GalE}), it allows us to prove a stronger finiteness statement for the curves appearing in the family of double covers $\widetilde{\mathcal{C}} \rightarrow U$. 
\end{rem} 
\printbibliography
\end{document}